\documentclass[12pt,a4paper]{article}
\usepackage{hyperref}
\usepackage[textformat=period]{caption}
\usepackage[english,british]{babel}
\usepackage[utf8]{inputenc}
\renewenvironment{abstract}{\begin{center}\begin{minipage}{0.7\linewidth}\noindent\textsc{Abstract} \hspace{1em}}{\end{minipage}\end{center}}
\newcommand{\half}{\frac{1}{2}}
\newcommand{\threeh}{\frac{3}{2}}
\newcommand{\rk}{\mathrm{rk}\,}
\title{Computing the unknotting numbers of certain pretzel knots}
\usepackage{amsmath,amsthm,amssymb,amsfonts}
\usepackage{graphicx}
\newcommand{\dstruck}[1]{\mathbb{#1}}

\newcommand{\Zx}{\dstruck{Z}}

\newtheorem{thm}{Theorem}
\ifx\thechapter\undefined
   \newtheorem{lem}{Lemma}[section]
\else
   \newtheorem{lem}{Lemma}[chapter]
\fi
\newtheorem{prp}[lem]{Proposition}
\newtheorem{dfn}[lem]{Definition}
\newtheorem{rem}[lem]{Remark}

\newtheorem{cnj}[lem]{Conjecture}

\newcommand{\coker}{\mathrm{coker}\,}
\author{Seph Shewell Brockway\\seph@undertherose.org.uk}
\begin{document}
\maketitle

\begin{abstract}
We compute the unknotting number of two infinite families of pretzel knots, $P(3,1,\dots,1,b)$ (with $b$ positive and odd and an odd number of $1$s) and $P(3,3,3c)$ (with $c$ positive and odd). To do this, we extend a technique of Owens using Donaldson's diagonalization theorem, and one of Traczyk using the Jones polynomial, building on work of Lickorish and Millett.
\end{abstract}

\section{Introduction}
\label{sec:intro}
The unknotting number $u(K)$ of a knot $K$ is the minimal number of crossing changes (whereby one strand of the knot is passed through another) required to transform $K$ into the unknot. Any diagram of $K$ can be used to compute such an unknotting sequence for $K$, and thereby place an upper bound on $u(K)$. Calculating $u(K)$ exactly, or even computing a lower bound, is in general a hard problem.

\begin{figure}\centering\includegraphics*{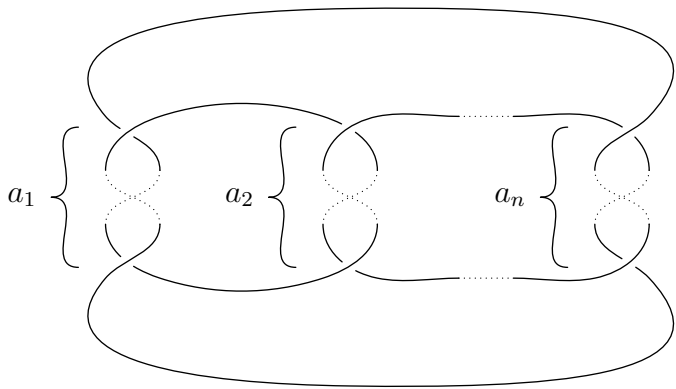}\caption{The pretzel link $P(a_1,\dots,a_n)$}\label{fig:pretz}\end{figure}

The pretzel link $P(a_1,\dots,a_n)$ with $a_i\in\Zx\setminus 0$ for all $i$ is the link shown in Figure \ref{fig:pretz}, with $a_i<0$ representing crossings in the opposite direction to that shown. Observe that $P(a_1,\dots,a_n)$ is a knot when $n$ and all of the $a_i$ are odd, and also when exactly one of the $a_i$ is even; every pretzel knot is of one of these two types.

We will make use of the knot signature $\sigma(K)$, originally defined in terms of a Seifert surface (an orientable surface embedded in $S^3$ bounded by $K$; see e.g. \cite{lickorish97}) with $\sigma(\mbox{unknot})=0$. It is well known (see e.g. \cite{cl86}) that if $K_-$ is obtained from $K_+$ by changing a positive crossing then $\sigma(K_-)-\sigma(K_+)\in\{0,2\}$, so that, for any knot $K$, $u(K)\geq\half|\sigma(K)|$. The method of Gordon and Litherland \cite{gl78} shows easily that $\sigma(P(a_1,\dots,a_n))=n-1$ whenever the $a_i$ are all positive and odd.

That $u(P(3,1,3))=2$ was established by Lickorish \cite{lickorish85}. Owens \cite{owens10} later showed that $P(3,1,3)$ could not be unknotted by changing one negative and any number of positive crossings, and in a separate paper \cite{owens08} also showed that $u(P(3,1,1,1,3))=3$.

Traczyk \cite{traczyk98} used the Jones polynomial to show that $P(3,3,3)$ could not be unknotted by changing one positive and one negative crossing. Owens \cite{owens08} used this work and an obstruction from Heegaard Floer theory to show that $u(P(3,3,3))=3$.

We extend the techniques of Owens and Traczyk to establish the following two theorems.

\begin{thm}\label{thm:ratpretz}For $K=P(3,1,\dots,1,b)$ with an odd number $r$ of $1$s and $b$ positive and odd, $u(K)=\half(r+3)$. More generally, $K=P(a,1,\dots,1,b)$, with an odd number $r$ of $1$s and $a$ and $b$ positive and odd, cannot be unknotted by changing $\half\sigma(K)=\half(r+1)$ negative and any number of positive crossings.\end{thm}

\begin{thm}\label{thm:pretz333u}We have $u(P(3,3,3c))=3$ for $c$ positive and odd. In general, for $K=P(3a,3b,3c)$ with $a$, $b$ and $c$ positive and odd, the unknotting number $u(K)\geq3$.\end{thm}

Both of these results are special cases of the following conjecture of Jablan and Sazdanović \cite{js07}.

\begin{cnj}\label{cnj:pretz}For $n$ odd and $a_1,\dots,a_n$ positive and odd, and with $a_1\leq a_2\leq\dots\leq a_n$, $$u(P(a_1,\dots,a_n))=\frac{\sum_{i=1}^{n-1}a_i}{2}\mbox.$$\end{cnj}

This paper is an expanded excerpt from a master's thesis \cite{shewellbrockway13} prepared at the University of Glasgow under the supervision of Brendan Owens, to whom many thanks are due for his advice and support.

\section{Unknotting rational pretzel knots}
\label{sec:ratpretz}
\setcounter{thm}{0}
In this section we consider pretzel knots of the form $P(a,1,\dots,1,b)$, with $a$ and $b$ odd and at least $3$, and an odd number $r$ of $1$s; knots of this form are also part of the family known as rational or two-bridge knots. The technique in this section was established for another family of rational knots, including $P(3,1,3)$, by Owens \cite{owens10}, who also established the case of $P(3,1,1,1,3)$ by other methods in \cite{owens08}.

We will require the following definition.

\begin{dfn}\label{dfn:halfintsurg}A bilinear form $q$ on some free abelian group $M$ of rank $2m$ is of \textbf{half-integer surgery type} if it has a matrix representation $$Q=\left(\begin{array}{c c}2I&I\\I&*\end{array}\right)$$ over some basis $\{x_1,\ldots,x_m,y_1,\ldots,y_m\}$. \cite[Definition 1]{owens10}\end{dfn}

We also make use of the following corollary of \cite[Theorem 3]{owens08} (see also \cite[Theorem 2]{owens10}, \cite[Theorem 1]{os13}).

\begin{prp}\label{prp:owens}Let $K$ be a knot with signature $\sigma$, and suppose that $K$ can be unknotted by changing $p$ positive and $n$ negative crossings, with $n=\half\sigma$. Then the branched double cover $Y$ of $S^3$ over $K$ bounds some smooth, oriented, positive-definite $4$-manifold $X$, with intersection form of half-integer surgery type.\end{prp}

See, e.g., \cite[Section 1.2]{gs99} for the definition of the intersection form $q_X$.

Recall that a lattice over the integers is a free abelian group $L$ equipped with a non-singular bilinear form $\cdot:L\otimes L\to\Zx$, and that given a sublattice $M$ of $L$, the orthogonal complement of $M$ is the sublattice $M^\perp=\{l\in L:l\cdot m=0,\forall m\in M\}$. Let $\Zx^m$ denote the free abelian group on generators $e_1,\dots,e_m$, with the bilinear form defined by $e_i\cdot e_j=\delta_{ij}$ (the Kronecker delta). For convenience we introduce the notation $\Lambda_X$ for the lattice $(H_2(X)/\mathrm{Tor}\,H_2(X),q_X)$.

\begin{thm}[restated]For $K=P(3,1,\dots,1,b)$ with an odd number $r$ of $1$s and $a$ positive and odd, $u(K)=\half(r+3)$. More generally, $K=P(a,1,\dots,1,b)$, with an odd number $r$ of $1$s and $a$ and $b$ positive and odd, cannot be unknotted by changing $\half\sigma(K)=\half(r+1)$ negative and any number of positive crossings.\end{thm}
\begin{proof}
Let $K=P(a,1,\dots,1,b)$, with $a$, $b$ and $r$ as above, and let $Y$ be the twofold branched cover of $S^3$ over $K$. A well-known result \cite[]{gl78} states that $Y$ is obtained as the boundary of a smooth $4$-manifold $Z$ with intersection form $q_Z$ equal to the Goeritz form $g_K$, having matrix representation $$Q_Z=\left(\begin{array}{c c c}A&0&\alpha\\0&B&\beta\\\alpha^\tau&\beta^\tau&-r-2\end{array}\right)\mbox,$$ where $A$ and $B$ are respectively $(a-1)\times (a-1)$ and $(b-1)\times (b-1)$ matrices of the form $$\left(\begin{array}{c c c c}-2&-1&0&0\\-1&\ddots&\ddots&0\\0&\ddots&\ddots&-1\\0&0&-1&-2\end{array}\right)$$ and $\alpha$ and $\beta$ are column vectors of the form $(-1,0,\dots,0)$. From Sylvester's criterion, we see that $Z$ is negative-definite.

We have two $4$-manifolds, $X$ and $Z$, with diffeomorphic boundaries. Consider the manifold $-Z$, with boundary $-Y$ and intersection form $q_{-Z}=-q_Z$. We can join $X$ and $-Z$ along their boundaries such that their orientations are preserved. We denote this manifold $W$. Note that in general $W$ is not unique: we have to make a choice of diffeomorphism of the boundaries; however, it doesn't matter which one we choose for the purposes of this argument.

We have established that $W$ is a closed, smooth $4$-manifold. Consider the Mayer--Vietoris sequence $$H_2(Y)\to H_2(X)\oplus H_2(-Z)\stackrel\phi\to H_2(W)\to H_1(Y)\mbox.$$ First, note that since the map $\phi$ is induced by the inclusion maps of $X$ and $-Z$ into $W$, it preserves intersection forms. Since $K$ is a knot, $H_2(Y)$ is trivial, so $\phi$ is a monomorphism. The cokernel $\coker\phi\subseteq H_1(Y)$ is finite since $K$ is a knot \cite{lickorish97}. We conclude from this that $W$ is positive-definite.

Donaldson's diagonalization theorem \cite{donaldson87} tells us that there exists a basis such that the intersection form $q_W$ has the $m\times m$ matrix representation $I_m=\mathrm{diag}(1,\dots,1)$. In terms of lattices, this means that we can embed $\Lambda_{-Z}$ in $\Zx^m\cong\Lambda_W$. Let $\Zx^m$ be generated by $e_1,\dots,e_m$ as above, and let $\Lambda_{-Z}$ have a basis $$\{\xi_1,\dots,\xi_{a-1},\eta_1,\dots,\eta_{b-1},\zeta\}\mbox,$$ over which $-q_Z$ has the matrix representation shown above with reversed signs. Up to changes of sign and permutations of the $e_i$, such an embedding must have the form \begin{align*}\xi_i&\mapsto e_i+e_{i+1}\\\eta_j&\mapsto e_{a+j}+e_{a+j+1}\mbox,\end{align*} but $\zeta$ does not embed uniquely.

We know that $\Lambda_X$ must embed as a finite-index sublattice into $\Lambda_{-Z}^\perp$. Since $\rk\Lambda_{-Z}=a+b-1$, it follows that $\rk \Lambda_{-Z}^\perp=m-a-b+1$. A finite-index sublattice of half-integer surgery type must therefore have $\half(m-a-b+1)$ generators $x_i$ with $x_i\cdot x_j=2\delta_{ij}$. Any element of $\Lambda_{-Z}^\perp$ whose expression involves a non-zero multiple of $e_i$, with $1\leq i\leq a$, must contain some multiple of $e_1-e_2+e_3-\dots+e_a$ by the definition of the orthogonal complement. Similarly, if $e_{a+j}$, with $1\leq j\leq b$, is involved in the expression we have to include $e_{a+1}-e_{a+2}+\dots+e_{a+b}$. Therefore these elements with square $2$ must come out of the sublattice of $\Lambda_W$ spanned by $e_{a+b+1},\dots,e_m$, of which there are $m-a-b-1$; for brevity we write $g_i=e_{a+b+i}$. First let $x_1=g_1+g_2$. We can't let $x_2=g_1-g_2$ because in that case $x_1\cdot y_1\equiv x_2\cdot y_1$ modulo $2$, and we need $x_1\cdot y_1=1$ and $x_1\cdot y_2=0$. Therefore we have to set $x_2=g_3+g_4$, $x_3=g_5+g_6$ and so on. Therefore the greatest number of $x_i$ we can embed in $\Lambda_{-Z}^\perp$ is $\half(m-a-b-1)=\half(m-a-b+1)-1$. The second part of the result follows, and since for $P(3,1,\dots,1,b)$ we have an explicit unknotting sequence of $\half(r+3)$ (negative) crossing changes (see Figure \ref{fig:3111bunknot}) we also obtain the first part.
\end{proof}
\begin{figure}\centering\includegraphics*{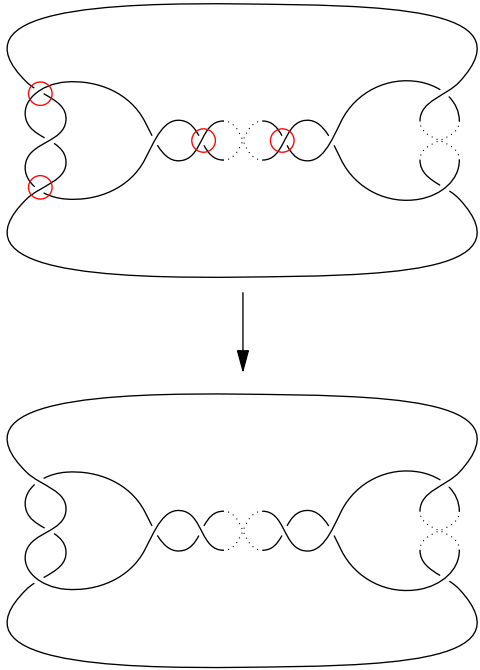}\caption{An unknotting sequence for $P(3,1,\dots,1,b)$}\label{fig:3111bunknot}\end{figure}

\begin{rem}\label{rem:os}If \cite[Theorem 1]{os13} is used in place of Proposition \ref{prp:owens}, then we obtain versions of Theorem 1 with the $4$-ball crossing number $c^*(K)$, the concordance unknotting number $u_c(K)$ or the slicing number $u_s(K)$ (see \cite{os13} for definitions of all of these) in place of the unknotting number.\end{rem}

\section{Unknotting more pretzels}
\label{sec:morepretz}
\begin{figure}\centering\includegraphics*{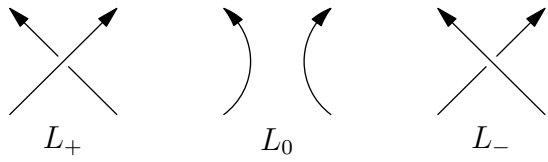}\caption{A skein triple}\label{fig:skein}\end{figure}
Recall that the Jones polynomial $V(L)$ is an oriented link invariant which takes values in the ring $\Zx[q^{\pm\half}]$ of Laurent polynomials in a single indeterminate $q^\half$ with integer coefficients. The Jones polynomial is defined by $V(O)=1$, where $O$ denotes the unknot, and the skein relation $$(q^\half-q^{-\half})V(L_0)=q^{-1}V(L_+)-qV(L_-)$$ for a skein triple of links differing only inside a $3$-ball as shown in Figure \ref{fig:skein}. A thorough treatment of the Jones polynomial, including a proof that this relation is indeed well-defined and sufficient to compute the Jones polynomial of any oriented link, may be found in Chapter 3 of \cite{lickorish97}.

Using the Jones polynomial in addition to the above results obtained using Donaldson's diagonalization theorem, it is possible to compute a lower bound on the unknotting number of $P(3a,3b,3c)$, giving an explicit value for $P(3,3,3c)$. We will require the following result of Lickorish and Millett.

\begin{prp}\label{prp:licmil}For any $r$-component link $L$, $V(L;\omega)=(-1)^si^{r-1}(i\sqrt3)^d$, where $d$ is the nullity of the modulo-$3$ reduction of the symmetrized Seifert form of $L$ and $\omega=e^{i\pi/3}$. \cite[Theorem 3]{lm86}\end{prp}

\begin{lem}\label{lem:jompretz1} Let $K=P(3a,3b)$, where $a$ and $b$ are positive and odd. Then $V(K;\omega)=-\sqrt3$.\end{lem}
\begin{proof}

The symmetrized Seifert form corresponding to the pretzel link $P(3a,b)$ has matrix representation $\hat S=(3a+b)$. Denote its modulo-$3$ version $\hat S_3=(\beta)$. This has nullity $1$ if $\beta=0$, that is if $3|b$, and nullity $0$ otherwise, and we apply Proposition \ref{prp:licmil}.

Define a skein triple $(L_+,L_-,L_0)$ by $L_+=P(3a,3b-2)$, $L_-=P(3a,3b)$ and $L_0=O$. Noting that $L_\pm$ have two components, we have $$V(L_+;\omega)=(-1)^{s_+}i$$ $$V(L_-;\omega)=-(-1)^{s_-}\sqrt3$$ $$V(L_0;\omega)=1\mbox.$$ We can substitute $q=\omega$ into the Jones skein relation, although care must be taken, as $\omega^\half$ has two possible values. Here we take $\omega^\half=e^{i\pi/6}$, but we could equally take $\omega^\half=e^{7i\pi/6}$; the argument in the latter case is entirely parallel to the one given here. In any case, we have \begin{align*}1=\frac{\omega^\half-\omega^{-\half}}{i}&=\pm\omega^{-1}-\omega(-1)^{s_-}i\sqrt3\\\frac{\pm\omega^{-1}-1}{i\omega\sqrt3}&=(-1)^{s_-}\mbox.\end{align*}

Take the $\pm$ sign first to be positive. This yields $$(-1)^{s_-}=\frac{-\omega}{i\omega\sqrt3}=\frac{i}{\sqrt3}$$ which obviously is not satisfied for any $s_-$.

Now let the $\pm$ be negative. Here, $$(-1)^{s_-}=\frac{-\omega^{-\half}\sqrt{3}}{i\omega\sqrt3}=i\omega^{-\threeh}=1\mbox.$$

We conclude that $s_-\equiv 0$ modulo $2$, so that $V(P(3a,3b);\omega)=-\sqrt3$ as required.
\end{proof}

\begin{lem}\label{lem:jompretz2} With $K=P(3a,3b,3c)$ with $a$, $b$ and $c$ positive and odd, $V(K;\omega)=3$.\end{lem}
\begin{proof}
The symmetrized Seifert form corresponding to $P(3a,3b,c)$ has matrix representation $$\hat S=\left(\begin{array}{c c}3(a+b)&-3b\\-3b&3b+c\end{array}\right)$$ with modulo-$3$ version $$\hat S_3=\left(\begin{array}{c c}0&0\\0&\gamma\end{array}\right)\mbox.$$ This has nullity $2$ if $\gamma=0$, that is if $3|c$, and nullity $1$ otherwise.

Define a skein triple by $L_+=P(3a, 3b, 3c-2)$, $L_-=P(3a, 3b, 3c)$, $L_0=P(3a, 3b)$. We have $$V(L_+;\omega)=(-1)^{s_+}i\sqrt3$$ $$V(L_-;\omega)=-3(-1)^{s_-}$$ $$V(L_0;\omega)=-\sqrt3\mbox.$$ The skein relation gives \begin{align*}-i\sqrt3&=\pm i\omega^{-1}\sqrt3+3\omega(-1)^{s_-}\\-\frac{i\sqrt3(1\pm\omega^{-1})}{3\omega}&=(-1)^{s_-}\mbox.\end{align*}

Take the $\pm$ sign to be positive. Thus $$(-1)^{s_-}=-\frac{3i\omega^{-\half}}{3\omega}=-i\omega^{-\threeh}=-1\mbox.$$

Taking the $\pm$ sign to be negative, $$(-1)^{s_-}=-\frac{\sqrt3i\omega}{3\omega}=-\frac{i}{\sqrt3}\mbox,$$ which again has no solutions.

Therefore $s_-\equiv 1$ modulo $2$, so $V(P(3a,3b,3c);\omega)=3$ as required.
\end{proof}

We now use the following result of Traczyk.

\begin{prp}\label{prp:traczyk}Let $K$ be a knot with $V(K;\omega)=(-1)^s(i\sqrt3)^d$ that can be transformed into the unknot by changing $n$ negative and $p$ positive crossings, such that $n+p=d$. Then $p\equiv s$ modulo $2$. \cite[Theorem 3.1]{traczyk98}\end{prp}

\begin{thm}[restated]We have $u(P(3,3,3c))=3$ for $c$ positive and odd. In general, for $K=P(3a,3b,3c)$ with $a$, $b$ and $c$ positive and odd, the unknotting number $u(K)\geq3$.\end{thm}
\begin{proof}
Since $P(3,1,3)$ can be transformed into $P(3a,3b,3c)$ by changing $3(a+b+c)-7$ positive and no negative crossings (to see this, it is more instructive to change negative crossings in the standard diagram of $P(3a,3b,3c)$ in order to reach $P(3,1,3)$), an unknotting sequence for $P(3a,3b,3c)$ with $n\leq1$ would induce an unknotting sequence for $P(3,1,3)$ with the same value of $n$, which is impossible by Theorem \ref{thm:ratpretz}. If $u=1$ then obviously $n\leq1$. This rules out the case $u=1$.

Assume that $u=2$. By Lemma \ref{lem:jompretz2}, $s\equiv1$ modulo $2$. Since $0\leq p\leq2$, Proposition \ref{prp:traczyk} tells us that $p\geq1$, so that $n\leq1$. As in the previous paragraph, Theorem \ref{thm:ratpretz} rules out an unknotting sequence for $P(3a,3b,3c)$ with $n\leq1$.

In the case where $a=b=1$, we have an explicit unknotting sequence of three crossing changes (Figure \ref{fig:333cunknot}), so $u=3$.

\begin{figure}\centering\includegraphics*{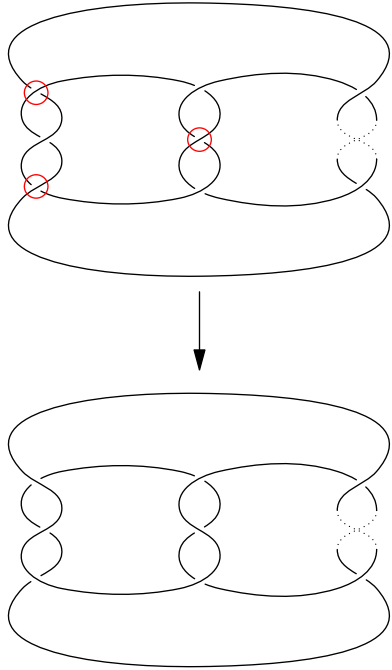}\caption{An unknotting sequence for $P(3,3,3c)$}\label{fig:333cunknot}\end{figure}
\end{proof}

\bibliography{maths}
\end{document}